\documentclass{amsart}

\usepackage{graphicx}
\usepackage{amsfonts}
\usepackage{amssymb}
\usepackage{amsmath}
\usepackage{xcolor} 
\usepackage{emptypage}

\usepackage[hidelinks]{hyperref} 
\hypersetup{
    colorlinks=true,
    linkcolor=red,
    citecolor=blue,
    filecolor=red,
   urlcolor=red}

\newtheorem{theorem}{Theorem}

\newtheorem{lemma}{Lemma}
\newtheorem{proposition}{Proposition}
\newtheorem{corollary}{Corollary}

\newcommand{\GG}{\mathbf{G}}
\newcommand{\cc}{\mathfrak{c}}
\newcommand{\hcc}{\hat{\mathfrak{c}}}
\newcommand{\tcc}{\tilde{\mathfrak{c}}}

\newcommand{\fconnect}[3]{{#1}\ {\overset{{#2}}{\searrow}}\ {#3}}
\newcommand{\bconnect}[3]{{#1}\ {\overset{{#2}}{\swarrow}}\ {#3}}
\newcommand{\dconnect}[3]{{#1}\ {\overset{{#2}}{\downarrow}}\ {#3}}
\newcommand{\boldord}[1]{\widehat{\boldsymbol{#1}}}

\begin{document}

\title{On closed Ramsey numbers of small countable ordinals}
\author{Necdet Duman}
\address{Department of Mathematics, Middle East Technical University, 06800, \c{C}ankaya, Ankara, Turkey\\
}
\email{necdet.duman@metu.edu.tr}
\author{Özge Gönül}
\address{Department of Mathematics, Middle East Technical University, 06800, \c{C}ankaya, Ankara, Turkey\\
}
\email{gonul.ozge@metu.edu.tr}
\author{Burak Kaya}
\address{Department of Mathematics, Middle East Technical University, 06800, \c{C}ankaya, Ankara, Turkey\\
}
\email{burakk@metu.edu.tr}
\author{Jayatra Saxena}
\address{Department of Mathematics, Middle East Technical University, 06800, \c{C}ankaya, Ankara, Turkey\\
}
\email{jayatra.saxena@metu.edu.tr}
\author{Y\.{I}\u{g}\.{I}than Tamer}
\address{Department of Mathematics, Middle East Technical University, 06800, \c{C}ankaya, Ankara, Turkey\\
}
\email{yigithan.tamer@metu.edu.tr}

\keywords{topological partition relations, Ramsey numbers, topological Ramsey numbers, closed Ramsey numbers}
\subjclass[2020]{03E02,03E10}

\begin{abstract} This paper is a contribution to the investigation of closed partition relations for pairs of countable ordinals. As our main result, we prove that
\[\omega^4 \cdot (n-2)+1 < R^{cl}(\omega \cdot n+1,3)<\omega^5\]
for every integer $n \geq 3$. This result significantly improves the existing upper and lower bounds for these closed Ramsey numbers. In addition, we prove that
\[\omega^{\theta}\nrightarrow_{cl} (\omega^{\alpha},3)^2\]
whenever $1 \leq \alpha \leq \theta<\omega_1$ satisfy $\theta < R(\alpha,3)$. This result asymptotically improves the existing lower bounds for $R^{cl}(\omega^n,3)$ and slightly strengthens the existing necessary condition for being a topological partition ordinal.
\end{abstract}

\maketitle

\tableofcontents

\section{Introduction}

As a natural extension of Ramsey's theorem \cite{Ramsey29} to Cantor's transfinite numbers, partition relations for ordinals and cardinals were introduced by Erd\H os and Rado in \cite{ErdosRado53} and \cite{ErdosRado56}, and have been thoroughly studied since then. In \cite{Baumgartner86}, Baumgartner extended this partition calculus to countable topological spaces.

Following Baumgartner, Caicedo and Hilton initiated an investigation of topological and closed Ramsey numbers for pairs of countable ordinals in \cite{CaicedoHilton17}. This investigation was recently contributed by Mermelstein in \cite{Mermelstein19} and \cite{Mermelstein20}, and by Kaya and Sağlam in \cite{KayaSaglam21}. This paper is another contribution to this systematic investigation. In order to be able to state the existing results and explain where the results of this paper fit, let us recall some basic notation and definitions.

Let $\alpha, \beta$ be ordinals endowed with their respective order topologies and let $X \subseteq \alpha$ be a topological subspace. The subspace $X$ is said to be \textit{order-homeomorphic} to $\beta$ if there exists an order-isomorphism $f: X \rightarrow \beta$ which is also a homeomorphism. In this case, we shall say that $X$ is a \textit{closed copy} of $\beta$.\footnote{We would like to warn the reader that $X$ being a closed copy of $\beta$ does \textit{not} imply that $X$ is a closed subset of $\alpha$, unless $\beta$ is a successor ordinal. Though, it does imply that $X$ is closed under taking supremum below $\sup(X)$, which is why this terminology borrowed from \cite{CaicedoHilton17} seems fit.}

Given ordinals $\alpha,\beta,\gamma$, we write
$\gamma \rightarrow_{cl} (\alpha,\beta)^2$ in the case that for every function $\cc: [\gamma]^2 \rightarrow \{0,1\}$ there exists a subspace $X \subseteq \gamma$ such that
\begin{itemize}
\item $X$ is a closed copy of $\alpha$ and $[X]^2 \subseteq \cc^{-1}(0)$, or
\item $X$ is a closed copy of $\beta$ and $[X]^2 \subseteq \cc^{-1}(1)$.
\end{itemize}
where $[X]^2=\{Y \subseteq X: |Y|=2\}$ denotes the set of  two element subsets of $X$. For simplicity, we shall say that $X$ is \textit{homogeneous of color {\color{red}red}} (respectively, \textit{{\color{blue}blue}}) whenever we have $[X]^2 \subseteq \cc^{-1}(0)$ (respectively, $[X]^2 \subseteq \cc^{-1}(1)$.) 
The closed Ramsey number $R^{cl}(\alpha,\beta)$ is defined to be the least ordinal $\gamma$ such that $\gamma \rightarrow_{cl} (\alpha,\beta)^2$ if such an ordinal exists at all.

In \cite{CaicedoHilton17}, among other interesting results, Caicedo and Hilton proved that $R^{cl}(\omega \cdot n +1, k+1)<\omega^{\omega}$ and $R^{cl}(\omega^2,k) \leq \omega^{\omega}$ for all integers $n,k \geq 1$. In \cite{Mermelstein19} and \cite{Mermelstein20}, after introducing the notion of a canonical coloring, Mermelstein computed the closed Ramsey numbers $R^{cl}(\omega \cdot 2,3)=\omega^3 \cdot 2$ and $R^{cl}(\omega^2,3)=\omega^6$. Using Mermelstein's machinery on canonical colorings, in \cite{KayaSaglam21}, Kaya and Sağlam showed that $\omega^2 \cdot n < R^{cl}(\omega+n,3) < \omega^2 \cdot (n+1)$, completely determining the leading term of $R^{cl}(\omega+n,3)$ in its Cantor normal form as $\omega^2 \cdot n$. The following summarizes the existing state of the art for $R^{cl}(\alpha,3)$ with $\omega < \alpha \leq \omega^2$.
\begin{align*}
&R^{cl}(\omega+1,3)=\omega^2+1 & \text{\cite{CaicedoHilton17}}\\
&R^{cl}(\omega+2,3)=\omega^2\cdot 2+\omega+2 & \text{\cite{CaicedoHilton17}}\\
\omega^2 \cdot n <\ &R^{cl}(\omega+n,3) < \omega^2 (n+1)\ \ \text{ for } n\geq 3& \text{\cite{KayaSaglam21}}\\
&R^{cl}(\omega \cdot 2,3)=\omega^3 \cdot 2 & \text{\cite{Mermelstein19}}\\
&R^{cl}(\omega^2,3)=\omega^6 & \text{\cite{Mermelstein20}}
\end{align*}
Even though the proof of \cite[Theorem 6.1]{CaicedoHilton17} indirectly provides some upper bounds for $R^{cl}(\omega \cdot n +1,3)$, which require the computation of various Ramsey numbers and closed pigeonhole numbers, these upper bounds seems to be already beaten by the upper bound $\omega^6$. The main objective of this paper is to somewhat fill the gap between $R^{cl}(\omega \cdot 2,3)$ and $R^{cl}(\omega^2,3)$ and prove the following.
\begin{theorem}\label{maintheorem} $\omega^4 < R^{cl}(\omega \cdot 2+1,3)<\omega^5$ and $\omega^4 \cdot (n-2)+1 < R^{cl}(\omega \cdot n+1,3)<\omega^5$ for every integer $n \geq 3$.    
\end{theorem}
Since we do not exactly pinpoint the leading term of $R^{cl}(\omega \cdot n+1,3)$ in its Cantor normal form, there remains work to be done to determine the exact asymptotic behavior of $R^{cl}(\alpha,3)$ for $\omega \leq \alpha \leq \omega^2$. Nevertheless, our result is still able to illustrate the interesting phenomenon that $R^{cl}(\alpha,3)$ jump from $\omega^3 \cdot 2$ to above $\omega^4$ and jump from below $\omega^5$ to $\omega^6$ as $\alpha$ approaches $\omega^2$.

In regard to the closed Ramsey numbers $R^{cl}(\alpha,3)$ where $\alpha \geq \omega^2$, we do not carry out an extensive analysis. On the other hand, as a byproduct of the approach that we are using, we are able to prove the following negative closed partition relation for powers of $\omega$.
\begin{theorem}\label{toppartmaintheorem} Let $1 \leq \alpha\leq \theta < \omega_1$. If ${\theta}<R(\alpha,3)$, then $\omega^{\theta}\nrightarrow_{cl} (\omega^{\alpha},3)^2$ where $R(\cdot,\cdot)$ denotes the classical ordinal Ramsey number.  
\end{theorem}
For finite powers of $\omega$, Theorem \ref{toppartmaintheorem} implies $R^{cl}(\omega^n,3) \geq \omega^{R(n,3)-1}+1$ which asymptotically strengthens the existing bound $R^{cl}(\omega^n,3) \geq \omega^{5n-4}$ given by \cite{Mermelstein20} because the finite Ramsey numbers $R(n,3)$ are known to have asymptotic order of magnitude $n^2/\log(n)$ \cite{Kim95}.

We shall also provide a necessary condition for being a topological partition ordinal. Recall that an ordinal $\alpha$ is said to be a \textit{partition ordinal} (respectively, a \textit{topological partition ordinal}) in the case that $\alpha$ satisfies the partition relation $\alpha \rightarrow (\alpha,3)^2$ (respectively, $\alpha \rightarrow_{top} (\alpha,3)^2$.)\footnote{Here the classical partition relation $\gamma \rightarrow (\alpha,\beta)^2$ (respectively, the topological partition relation $\gamma \rightarrow_{top} (\alpha,\beta)^2$) is defined similarly to the closed partition relation $\gamma \rightarrow_{cl} (\alpha,\beta)^2$, except that the desired red and blue homogeneous sets are required to be only order-isomorphic (respectively, homeomorphic) to $\alpha$ and $\beta$.}

In \cite[Question 11.2]{CaicedoHilton17}, Caicedo and Hilton asked whether or not there exists a countable topological partition ordinal greater than $\omega$. They noted that any such ordinal has to be of the form $\omega^{\omega^{\beta}}$ and that this question is equivalent to the version of the question in which $\rightarrow_{top}$ is replaced by $\rightarrow_{cl}$. An immediate corollary of Theorem \ref{toppartmaintheorem} is the following.

\begin{corollary}\label{toppartordcorollary} Let $1 \leq \alpha < \omega_1$. If $\omega^{\alpha}$ is a topological partition ordinal, then $\alpha$ is a partition ordinal.    
\end{corollary}

Combining Corollary \ref{toppartordcorollary} with the necessary conditions for being a partition ordinal given in \cite{GalvinLarson74} and \cite[Theorem 29.(3)]{Schipperus10}, one obtains that any countable topological partition ordinal greater than $\omega$ is of the form
\[\omega^{\omega},\ \omega^{\omega^2} \text{ or }\ \omega^{\omega^{\omega^{\beta}}}\]
where $\beta$ is the sum at most $3$ indecomposable ordinals. This result slightly improves the necessary condition stated in \cite[Section 11]{CaicedoHilton17}. All that said, we have not been able to establish the existence of a countable topological partition ordinal greater than $\omega$. Thus \cite[Question 11.2]{CaicedoHilton17} remains open.

Finally, we would like to point out that, Theorem \ref{maintheorem} and Theorem \ref{toppartmaintheorem} still hold if one replaces $R^{cl}(\cdot,\cdot)$ and $\rightarrow_{cl}$ by $R^{top}(\cdot,\cdot)$ and $\rightarrow_{top}$ respectively. This is because ordinals of the form $\omega^{\beta}$ with $\beta \neq 0$ and $\omega \cdot n +1$ with $n \in \mathbb{Z}^+$ are order-reinforcing \cite[Corollary 2.17]{Hilton16}. This paper is organized as follows.

In Section \ref{sectionpreliminaries}, after recalling necessary definitions and notation, we briefly overview the machinery of canonical colorings introduced by \cite{Mermelstein19} and their interpretation as directed graphs.

In Section \ref{sectionupperbound}, after introducing the necessary finite Ramsey theoretic notions that will be used in the proof, using the directed graph approach provided in Section \ref{sectionpreliminaries}, we prove a positive closed partition relation that establishes the upper bounds in Theorem \ref{maintheorem}.

In Section \ref{sectionlowerbound}, after providing an overview of the existing graph-based approach to create colorings that witness various negative closed partition relations, we introduce matryoshka colorings as a stronger alternative. Using matryoshka colorings, we then prove negative closed partition relations that establish the lower bounds in Theorem \ref{maintheorem}.

In Section \ref{sectiontoppart}, we prove Theorem \ref{toppartmaintheorem} and Corollary \ref{toppartordcorollary} using a simple lemma regarding the set of Cantor-Bendixson ranks of elements of closed copies of ordinals.

\textbf{Acknowledgements.} This work was supported by the Scientific and Technological Research Council of Türkiye (TÜBİTAK) under the Project Number 124F353, of which the third author is the project manager and through which the remaining authors are funded as members of the project. The authors thank TÜBİTAK for its support.

\section{Preliminaries}\label{sectionpreliminaries}

The main machinery that we shall use in our upper bound argument is the notion of a canonical coloring introduced by \cite{Mermelstein19}. While we are not going to overview the canonization process of an arbitrary coloring in this paper and simply take the existence of canonical colorings as a blackbox, we still need to introduce some preliminary notions and notation in order to be able to state the definition. In this section, we shall briefly recall these prerequisites and the main machinery, together with a directed graph interpretation.

\subsection{Ordinals below $\omega^{\omega}$ as forests and their large subsets}\label{ordinalpreliminaries} In this subsection, following \cite{CaicedoHilton17} and \cite{Mermelstein19}, we will recall how to represent ordinals below $\omega^{\omega}$ as forests, that is, acyclic graphs.

Let $\gamma$ be a non-zero ordinal with the Cantor normal form
$$ \gamma=\omega^{\alpha_1} \cdot k_1 + \omega^{\alpha_2} \cdot k_2 + \dots + \omega^{\alpha_m} \cdot k_m$$
where $\alpha_1 > \alpha_2 > \dots > \alpha_m$ are ordinals and $0<k_1,k_2,\dots,k_m <\omega$ are positive integers. \textit{The Cantor-Bendixson rank of $\gamma$} is defined to be the ordinal $CB(\gamma)=\alpha_m$. For the case $\gamma=0$, we define $CB(0)=0$. We also define $L(\gamma)=k_m$ and $L(0)=1$. Consider the partial order relation $<^*$ on the class of ordinals defined by
\[ \beta <^* \alpha \text{ if and only if } \alpha=\beta+\omega^{\delta} \text{ for some } \delta>CB(\beta)\]
We shall write $\beta \triangleleft^* \alpha$ in the case that $\alpha$ is the necessarily unique immediate successor of $\beta$. This partial order relation first appeared in \cite{CaicedoHilton17}. It is not difficult to check that, by joining the vertices satisfying $\beta \triangleleft^* \alpha$, ordinals below $\omega^{\omega}$ can be represented as forests, an example illustration of which shall be given later.

By ungrouping the terms in the Cantor normal of $\gamma$, we can also express $\gamma$ as $$\gamma=\omega^{\beta_1}+\omega^{\beta_2}+\dots+\omega^{\beta_n}$$where $\beta_1 \geq \beta_2 \geq \dots \geq \beta_n$ are ordinals. In this case, for each ordinal $\alpha \leq \gamma$, we set
$$CNF_{\gamma}(\alpha)=\min\{1 \leq k \leq n\ |\ \alpha \leq \omega^{\beta_1}+\omega^{\beta_2}+\dots+\omega^{\beta_k}\}$$
In the case that $\gamma<\omega^\omega$, all these notions can be interpreted pictorially in the forest representation of $\gamma$ as follows. An ordinal $\alpha<\gamma$ resides
\begin{center}
at the $L(\alpha)^{\text{th}}$ place of the $CB(\alpha)^{\text{th}}$ level of the $CNF_{\gamma}(\alpha)^{\text{th}}$ tree.\footnote{In the case that $\gamma$ is a limit ordinal, while computing $CNF_{\gamma}(\alpha)$ by counting on which tree $\alpha$ resides, we count the last $\omega$-many trees as a single tree as if there were a vertex above them all.}\footnote{The case $\alpha=0$ is an exceptional case where we pretend that $0$ and $1$ are both in the $1^{\text{st}}$ place of the $0^{\text{th}}$ level of the $1^{\text{st}}$ tree. This slight inconvenience does not cause any harm in our arguments.}
\end{center}
Given an ordinal $\alpha \leq \gamma<\omega^{\omega}$ and $i<\omega$, consider the sets
\[T(\alpha)=\{\beta: \beta <^* \alpha\} \cup \{\alpha\}\ \ \text{ and }\ \ T^{=i}(\alpha)=\{\beta \in T(\alpha): CB(\beta)=i\}\]
which respectively give the vertices of the subtree below the vertex $\alpha$ and its $i$-th level. It is not difficult to verify that the subspace $T(\alpha) \subseteq \gamma$ is order-homeomorphic to the ordinal $\omega^{CB(\alpha)}+1$ whenever $CB(\alpha) \neq 0$. For later use, let us define a special collection of ``large subsets" of the levels $T^{=i}(\alpha)$. These large subsets first appeared in \cite{Mermelstein19} as a part of the canonization process of a coloring.

We first handle the powers of $\omega$. Let $k,r \in \mathbb{N}$. For every $0 \leq m \leq k$, we define the sets $F(\omega^k)^r_m$ recursively as follows.
\[F(\omega^k)^r_m=\begin{cases} \{ \omega^k\} & \text{ if } m=k,\\ 
\bigcup_{\alpha \in F(\omega^k)^r_{m+1}} \{ \beta \in \omega^k :\ \beta \triangleleft^* \alpha,\ \ L(\beta)>r  \}& \text{ if } m< k. \end{cases} \]
Since $T(\alpha)$ is order-isomorphic to $\omega^{CB(\alpha)}+1$ whenever $CB(\alpha) \neq 0$, this definition can be naturally extended to all ordinals $\alpha<\omega^\omega$ in the following manner: For every $\alpha<\omega^\omega$ with $CB(\alpha) = 0$, the set $F(\alpha)^r_0$ is defined as $F(\alpha)^r_0=\{\alpha\}$. For every $\alpha<\omega^\omega$ with $CB(\alpha) \neq 0$, the set $F(\alpha)^r_m$ is defined as
\[ F(\alpha)^r_m=\rho^{-1}\left[F\left(\omega^{CB(\alpha)}\right)^r_m\right]\]
where $\rho: T(\alpha) \rightarrow \omega^{CB(\alpha)}+1$ is the unique order-isomorphism. Before we conclude this subsection, we refer the reader to Figure \ref{forestrepresentation} for an illustration of these notions on the ordinal $\gamma=\omega^3+\omega^2 \cdot 2$.

\begin{figure}
    \centering
    \includegraphics[width=1.0\textwidth]{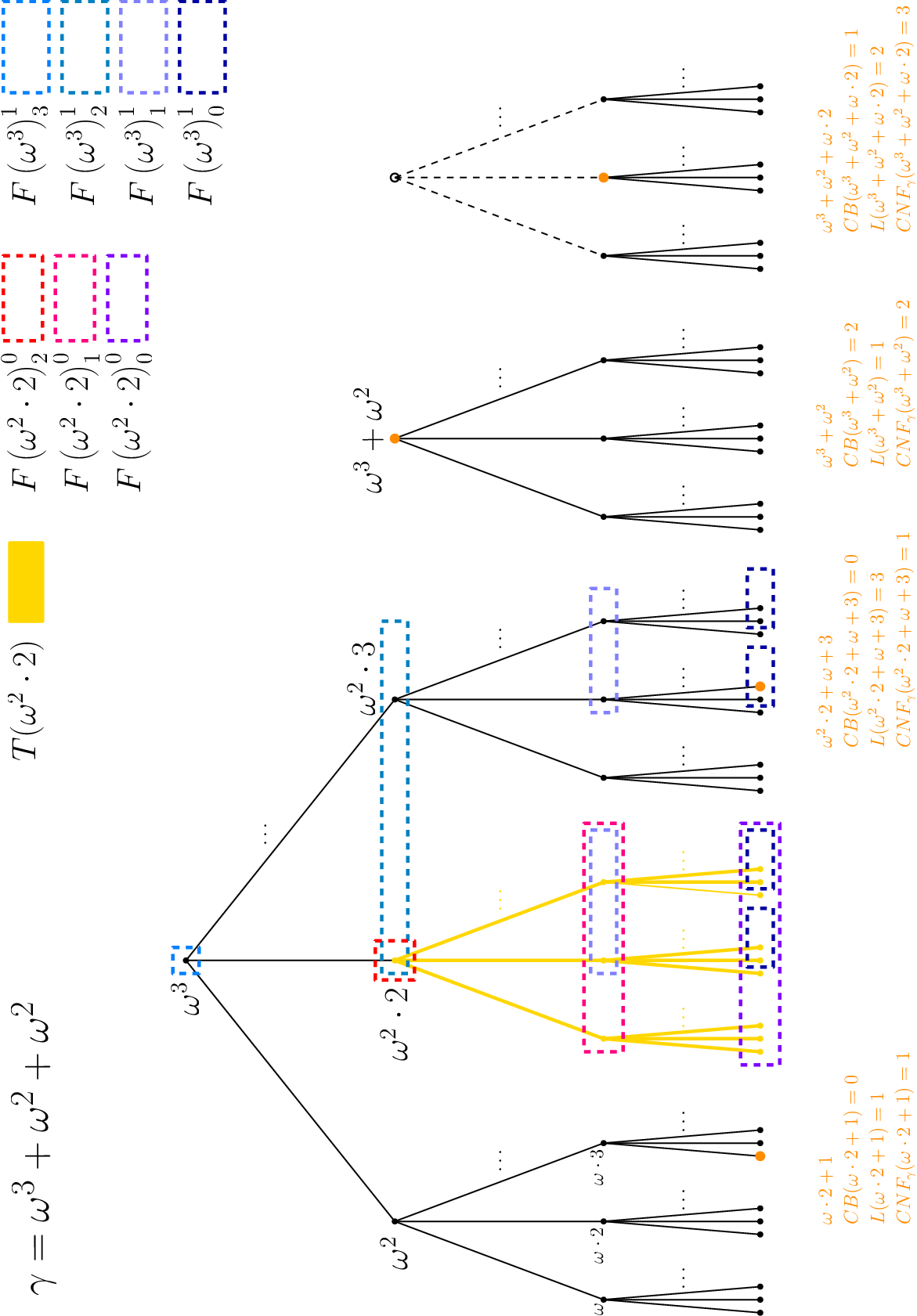}
    \caption{A representation of the ordinal $\gamma=\omega^3+\omega^2 \cdot 2$ as a forest}
    \label{forestrepresentation}
\end{figure}

\subsection{Canonical colorings}\label{specialtypescoloring} In this subsection, we recall the definition of a canonical coloring first introduced in \cite{Mermelstein19}. A few words of warning are in order. First, in the definition that we shall present, following the presentation style of \cite[Section 2.2]{KayaSaglam21} from which we copied various parts verbatim, we unpacked some auxiliary definitions originally present in \cite{Mermelstein19} because these auxiliary definitions are not required for the purposes of this paper. Second, we restrict our definition to only successor ordinals so that we avoid potential trivialities in inequalities arising from the non-existence of the top vertex of the last tree in the forest representation.

Let $\gamma<\omega^{\omega}$ be a successor ordinal, say,
$$\gamma=\omega^{m_1}+\omega^{m_2}+\dots+\omega^{m_n}+1$$ where $m_1 \geq m_2 \geq \dots \geq m_n$ are natural numbers.

A (pair) \textit{coloring} of $\gamma$ with two colors is a map $\cc: [\gamma]^2 \rightarrow \{0,1\}$. A coloring $\cc: [\gamma]^2 \rightarrow \{0,1\}$ is said to be \textit{normal} if, for all $\beta_1<^*\beta_2<\gamma$, the color of the pair $\{\beta_1,\beta_2\}$ only depends on $CB(\beta_1)$, $CB(\beta_2)$ and $CNF_\gamma(\beta_2)$, that is, there exists a function $\hcc(i,j,k)$ defined for $1 \leq i \leq n$ and $0 \leq k<j \leq m_i$ that does not depend on $\beta_1$ and $\beta_2$ such that
\[
\mathfrak{c}(\{\beta_1, \beta_2\}) = \hcc(CNF_\gamma(\beta_2), CB(\beta_2), CB(\beta_1))
\]
Pictorially speaking, this means that the color of a pair of ordinals that are $<^*$-related within the same tree in the forest representation of $\gamma$ solely depends on the levels.

A normal coloring $\cc:[\gamma]^2 \rightarrow \{0,1\}$ is called \textit{canonical} if the following additional conditions are met.
\begin{itemize}
\item[(a)] For every $\alpha < \gamma$, there exists some $r \in \mathbb{N}$ such that
\begin{center}
    for each $\theta<\gamma$ and ${\ell} \leq CB(\theta)$\footnote{Here, we require ${\ell} < CB(\theta)$ in the case $\alpha=\theta$. Moreover, the reader is also referred to ``Important remark" in \cite[Section 2.2]{KayaSaglam21} in regard to this inequality being non-strict.} there exists a color $c_{\alpha}(\theta,{\ell}) \in \{0,1\}$ with
\[ \{\beta \in T^{={\ell}}(\theta):\ \mathfrak{c}(\{ \alpha, \beta \}) =c_{\alpha}(\theta,{\ell}) \} \supseteq F(\theta)^r_{\ell}\]
\end{center}
In other words, for every $\alpha<\gamma$, the singleton coloring map $\cdot \mapsto \cc(\{\alpha,\cdot\})$ is constant on a ``large subset" of every level of every tree.
\item[(b)] For every $\alpha,\beta<\gamma$ such that $CNF_{\gamma}(\alpha)=CNF_{\gamma}(\beta)$ and $CB(\alpha)=CB(\beta)$, we have
\begin{center}
    $c_{\alpha}(\omega^{m_1}+\dots+\omega^{m_k},{\ell})=c_{\beta}(\omega^{m_1}+\dots+\omega^{m_k},{\ell})$
\end{center}
for every $1 \leq k \leq n$ with $k \neq CNF_{\gamma}(\alpha)$ and for every $0 \leq {\ell}\leq m_k$. That is, the color $c_{\alpha}(\theta,\ell)$ defined in Item (a) is solely determined by $CNF_{\gamma}(\alpha)$ and $CB(\alpha)$.
\end{itemize}
It follows from Item (b) that, given a canonical coloring $\cc:[\gamma]^2 \rightarrow \{0,1\}$, one can construct a function $\tcc(i,j;k,\ell)$ defined for $1 \leq k \neq i \leq n$ and $0 \leq j \leq m_i$ and $0 \leq \ell \leq m_k$ such that
\[ \tcc(i,j;k,{\ell})=c_{\alpha}(\omega^{m_1}+\dots+\omega^{m_k},{\ell})\]
where $\alpha \in \gamma$ is an arbitrary ordinal satisfying $CNF_{\gamma}(\alpha)=i$, $CB_{\gamma}(\alpha)=j$. It is not difficult to verify that, for a canonical coloring $\cc:[\gamma]^2 \rightarrow \{0,1\}$, the following conditions are equivalent.
\begin{itemize}
\item $\tcc(i,j;k,{\ell})=c$
\item For all $\alpha$ with $CNF_{\gamma}(\alpha)=i$ and $CB(\alpha)=j$, there exists $r \in \mathbb{N}$ such that $\cc(\{\alpha,\beta\})=c$ for every $\beta \in F\left(\omega^{m_1}+\dots+\omega^{m_k}\right)^r_{{\ell}}$.
\end{itemize}
The main theorem regarding canonical colorings is the following.
\begin{theorem}\cite[Proposition 3.11]{Mermelstein19}\label{maincanonicalcoloringtheorem} For every coloring $\cc: [\gamma]^2 \rightarrow \{0,1\}$ there exists a closed copy $I \subseteq \gamma$ of $\gamma$ such that
\begin{itemize}
    \item For all $\alpha,\beta\in I$, we have $\alpha<^* \beta$ iff $\rho_I(\alpha)<^* \rho_I(\beta)$ where $\rho_I: I \rightarrow \gamma$ is the unique order-homeomorphism, and
    \item the map $\cc_I: [\gamma]^2 \rightarrow \{0,1\}$ given by $\cc_I(\{\alpha,\beta\})=\cc\left(\left\{\rho_I^{-1}(\alpha),\rho_I^{-1}(\beta)\right\}\right)$ is a canonical coloring.
\end{itemize}    
\end{theorem}
We would like to remark that, as a result of Theorem \ref{maincanonicalcoloringtheorem}, in order to show that the closed partition relation $\gamma \rightarrow_{cl} (\alpha,\beta)^2$ holds, it is sufficient to verify that
\begin{center}
    for every \textbf{canonical} coloring $\cc: [\gamma]^2 \rightarrow \{0,1\}$ there exist a red homogeneous closed copy of $\alpha$ or a blue homogeneous closed copy of $\beta$.
\end{center}
This is because, given an arbitrary coloring $\cc: [\gamma]^2 \rightarrow \{0,1\}$, we can pass to a subset $I \subseteq \gamma$ as in the statement of Theorem \ref{maincanonicalcoloringtheorem} and we can pull back the homogeneous subsets with respect to $\cc_I$ via $\rho_I$.

\subsection{A directed graph approach}\label{directedgraphsection} In this subsection, by encoding the values of the maps $\tcc$ and $\hcc$, we shall attach a directed graph to each canonical coloring and explain how this directed graph can be used to deduce Ramsey theoretic information. We would like to note that this approach has already been used in \cite{Mermelstein19, Mermelstein20, KayaSaglam21} implicitly. Nevertheless, we would like to provide an explicit framework for this approach since our upper bound argument is more easily articulated in this framework.

Let $\cc:[\gamma]^2 \rightarrow \{0,1\}$ be a canonical coloring on the successor ordinal
$$\gamma=\omega^{m_1}+\omega^{m_2}+\dots+\omega^{m_n}+1$$ where $m_1 \geq m_2 \geq \dots \geq m_n$ are natural numbers. The canonical coloring $\cc$ induces a directed graph $\GG_{\cc}=(V,E)$ on the vertex set
\[V=\left\{V_i^j\right\}_{\substack{1 \leq i \leq n\\ 0 \leq j \leq m_i}} \text{ where } V_i^j=\{\alpha \in \gamma: CNF_{\gamma}(\alpha)=i \text{ and } CB(\alpha)=j\}\]
where the directed edge relation $E$ is defined as follows. For every $1 \leq k \neq i \leq n$ and $0 \leq j \leq m_i$ and $0 \leq \ell \leq m_k$,
\[E\left(V_i^j,V_k^{\ell}\right) \text{holds if and only if } \tcc(i,j;k,l)=1\]
and, for every  for $1 \leq i \leq n$ and $0 \leq k< j \leq m_i$,
\[E\left(V_i^j,V_i^k\right) \text{holds if and only if } \hcc(i,j,k)=1\]
In other words, $\GG_{\cc}=(V,E)$ is the directed graph on the set of levels of the trees in the forest representation of $\gamma$ in which the directed edges $\GG_{\cc}=(V,E)$ are placed in accordance with between which levels the maps $\hcc$ and $\tcc$ take the value {\color{blue} blue}.

We would like to remark that, while transitioning from the coloring itself to the associated directed graph, much information is lost. For example, colors of pairs of ordinals contained in the same tree but not related by $<^*$ are not reflected in the associated directed graph. However, at the cost of losing this local data, we are able to obtain a \textit{finite} combinatorial object that provides global data that is mostly sufficient to obtain reasonable lower and upper bounds. The only downside seems to be that the local data that is lost is presumably needed to be taken into account if one wishes to compute closed Ramsey numbers exactly.

In order to illustrate how these directed graphs can be used, suppose that there exist no red homogeneous closed copy of $\omega^2$ and no blue homogeneous (necessarily) closed copy of $3=\{0,1,2\}$ with respect to the canonical coloring $\cc$. Then the patterns in Figure \ref{forbiddenpatterns} are forbidden in the directed graph $\GG_{\cc}$.

\begin{figure}
    \centering
    \includegraphics[width=0.7\textwidth]{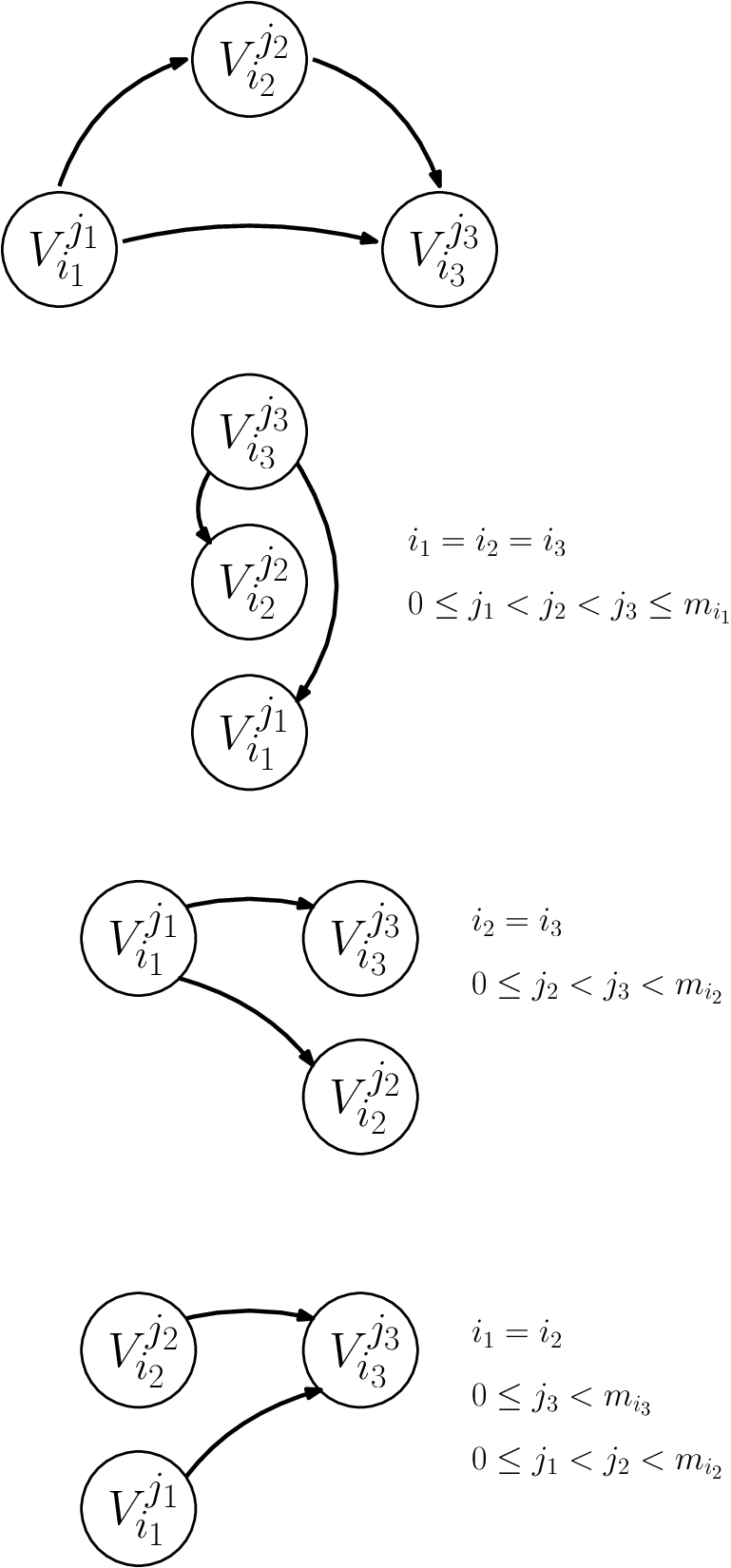}
    \caption{Four forbidden patterns in $\GG_{\cc}$ whenever there exist no red homogeneous closed copy of $\omega^2$ and no blue homogeneous copy of $3$ with respect to $\cc$.}
    \label{forbiddenpatterns}
\end{figure}

In order to see that the first pattern is forbidden, assume towards a contradiction that such $V_{i_1}^{j_1},V_{i_2}^{j_2}$ and $V_{i_3}^{j_3}$ exist in $\GG_{\cc}$. We are going to split into cases depending on which vertices are on the same tree of the forest representation of $\gamma$.

Suppose that $i_1 \neq i_2$, $i_1 \neq i_3$ and $i_2=i_3$. Let $\alpha \in V_{i_1}^{j_1}$. Since $\tcc(i_1,j_1; i_2,j_2)=1$ and $\tcc(i_1,j_1; i_3,j_3)=1$, there exist $r,r' \in \mathbb{N}$ such that
\begin{align*}
\cc(\{\alpha,\beta\})&=1\ \text{ for every }\beta \in F\left(\omega^{m_1}+\dots+\omega^{m_{i_2}}\right)^r_{{j_2}}\ \text{ and }\\
\cc(\{\alpha,\beta'\})&=1\ \text{ for every }\beta' \in F\left(\omega^{m_1}+\dots+\omega^{m_{i_3}}\right)^{r'}_{{j_3}}.
\end{align*}
We now set $\widehat{r}=\max\{r,r'\}+1$. Then, for any $\beta' \in F\left(\omega^{m_1}+\dots+\omega^{m_{i_3}}\right)^{\widehat{r}}_{{j_3}}$, we have
\[\cc(\{\alpha,\beta\})=\cc(\{\alpha,\beta'\})=\cc(\{\beta,\beta'\})=1\]
where $\beta$ is the unique ordinal in $F\left(\omega^{m_1}+\dots+\omega^{m_{i_2}}\right)^{\widehat{r}}_{{j_2}}$ with $\beta'<^*\beta$. This contradicts that there exists no blue homogeneous copy of $3$. The other cases are handled similarly using the definitions of $\tcc$ and $\hcc$.

The fact that the latter three patterns are forbidden is going to be central to our upper bound argument and we refer the reader to \cite[Lemma 4.3]{Mermelstein19} for a proof of this fact, of which the non-existence of these forbidden patterns is a restatement.

Before we conclude this section, we would like to mention that the non-existence of a blue homogeneous copy of $3$ does not prevent cyclic directed triangles of the form
\begin{center}
    \includegraphics[width=0.3\textwidth]{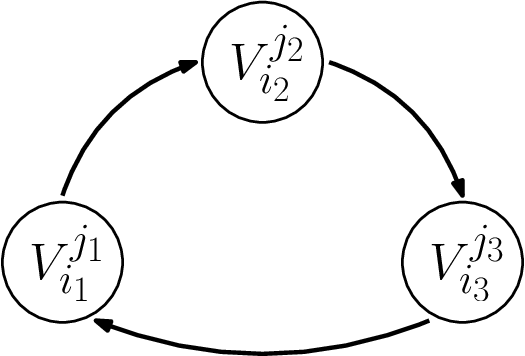}
\end{center}
from appearing in the directed graph $\GG_{\cc}$.

For example, consider the canonical coloring $\cc: [\omega \cdot 3+1]^2 \rightarrow \{0,1\}$ given by
\[\cc(\{\alpha,\beta\})=\begin{cases}
    1 &\text{ if } \alpha=j \text{ and }\beta=\omega +(j+1)+k \text{ for some } 0\leq j,k<\omega\\
    1 & \text{ if } \alpha=\omega +j \text{ and }\beta=\omega \cdot 2+(j+1)+k \text{ for some } 0\leq j,k<\omega\\
    1 & \text{ if } \alpha=\omega \cdot 2 +j \text{ and } \beta=(j+1)+k \text{ for some } 0\leq j,k<\omega\\
    0 & \text{ otherwise}
\end{cases}\]
It is not difficult to verify that $V_0^0$, $V_1^0$ and $V_2^0$ form such a cyclic directed triangle even though there exists no blue homogeneous set of size $3$.

\section{The upper bound}\label{sectionupperbound}

In this section, we shall prove that $R^{cl}(\omega \cdot n +1,3)<\omega^5$ for any integer $n \geq 2$ by showing that the closed partition relation
\[\gamma=\omega^4 \cdot M(n) +1 \rightarrow_{cl} (\omega \cdot n+1,3)^2\]
holds for a suitable choice of $M(n)$. In order to be able to define $M(n)$, we shall need various finite Ramsey theoretic notions, some of which have been incidentally used in \cite[Section 6]{CaicedoHilton17} for its upper bound argument as well.

\subsection{Coloring complete directed graphs} Let $n \geq 2$ be an integer. It turns out that there exists a positive integer $m$ such that for any edge coloring of the complete directed graph $K_m^*$, there exist a red homogeneous complete directed induced subgraph $K_n^*$ or a blue homogeneous copy of a transitive tournament on $3$ vertices, that is,
\begin{center}
    \includegraphics[width=0.25\textwidth]{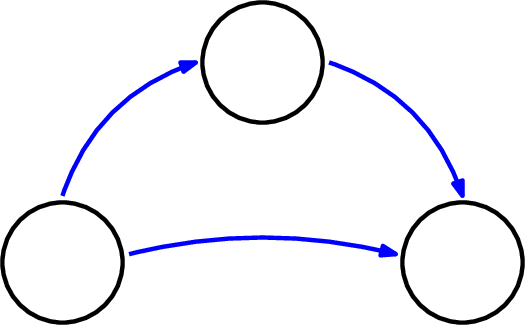}
\end{center}
The least such positive integer $m$ is denoted by $R(K_n^*,L_3)$. It is noted by Caicedo and Hilton in \cite[Section 6]{CaicedoHilton17} that the existence of the Ramsey number $R(K_n^*,L_3)$ was proven by Erd\H os and Rado in \cite{ErdosRado56,ErdosRado67}, although, in a slightly different form. As for how large these Ramsey number can be, Larson and Mitchell showed in \cite{LarsonMitchell97} that $R(K_n^*,L_3) \leq n^2$.

For our main argument, besides these directed graph Ramsey numbers, we shall need the following Ramsey-type result.

\begin{proposition}\label{avoidancenumberprop} Let $k \geq 3$ be an integer. Then there exists an integer $m \geq 1$ such that for every edge coloring of the directed complete graph $K_m^*$ with $3$ colors, there exists a subset $S \subseteq K_m^*$ of size $k$ such that
\begin{center}
    For every $x \in S$, the edges from vertices in $S$ to $x$ avoid at least one color.
\end{center}
\end{proposition}
\begin{proof} Set $m=R(k,k,k)$ where $R(k,k,k)$ denotes the classical Ramsey number. Given an edge coloring $c: E(K_m^*) \rightarrow \{0,1,2\}$ where the edge set is
\[E(K_m^*)=\{(a,b): a,b \in K_m^* \text{ and } a \neq b\}\]
consider the edge coloring $c': E(K_m) \rightarrow \{0,1,2\}$ of the undirected complete graph $K_m$ given by
\[c'\big(\{a,b\}\big)=\min \left(\{0,1,2\} \setminus \{c(a,b),c(b,a)\}\right)\]
By the definition of $R(k,k,k)$, there exists a subset $S \subseteq K_m$ of size $k$ such that the undirected edges between the vertices of $S$ are monochromatic, say, of color $i$. But then, by the construction of $c'$, for every $x \in S$, the edges from vertices in $S$ to $x$ in the directed graph $K_m^*$ avoid the color $i$.
\end{proof}
From now on, we shall denote the smallest such positive integer by $A_{3}(k)$. As seen in the proof of Proposition \ref{avoidancenumberprop}, we have $A_3(k) \leq R(k,k,k)$. It is also not difficult to check that $A_3(k) \geq 3^{k-2}/2^{k-1}$. While it is likely that this concept or an equivalent concept may have been already studied in the existing literature, we were unable to determine if this is indeed the case.

\subsection{A positive closed partition relation}

Let $n \geq 2$ be an integer. Set $M(n)=R(A_{3}(R(K_n^*,L_3)),3)$ and consider the ordinal
\[\gamma=\omega^4 \cdot M(n)+1\]
the forest representation of which consists of $M(n)$ trees of $5$ levels, each of which forms a closed copy of $\omega^4+1$. The main theorem of this section is the following.
\begin{theorem}\label{mainuppertheorem} $\gamma \rightarrow_{cl} (\omega \cdot n+1,3)^2$.   
\end{theorem}
\begin{proof} By the remarks at the end of Section \ref{specialtypescoloring}, it is sufficient to show that for every canonical coloring $\cc: [\gamma]^2 \rightarrow \{0,1\}$ there exist a red homogeneous closed copy of $\omega \cdot n +1$ or a blue homogeneous copy of $3$.

Let $\cc: [\gamma]^2 \rightarrow \{0,1\}$ be a canonical coloring. Assume towards a contradiction that there exist no such monochromatic copies with respect to $\cc$. Then, by the remarks in Section \ref{directedgraphsection}, the associated directed graph $\GG_{\cc}$ does not contain the forbidden patterns in Figure \ref{forbiddenpatterns} as an induced subgraph.

As a first step, we will find a set of indices $I_0$ for which the \textit{top vertices} form an independent set. Consider the restriction of the coloring $\cc$ to the set of top vertices
\[\left\{V_i^{4}\right\}_{i=1}^{M(n)}=\big\{\{\omega^4\},\{\omega^4 \cdot 2\},\dots,\{\omega^4 \cdot M(n)\}\big\}\]
Since there exists no blue homogeneous triangle with respect to $\cc$, by the definition of the Ramsey number $R(\cdot,3)$, there exists $I_0 \subseteq \{1,\dots,M(n)\}$ is of size $A_{3}(R(K_n^*,L_3))$ such that the set $\left\{V_{i}^4: i \in I_0\right\}$ is independent in $\GG_{\cc}$.

In the second step, we shall find a set of indices $I_1 \subseteq I_0$ and choose a non-top vertex from each corresponding column so that these top vertices are not connected to any chosen vertex. For each $i \in I_0$, choose three vertices $V_i^{k_0},V_i^{k_1},V_i^{k_2}$ to which there is no edge from $V_i^4$, where $0 \leq k_0<k_1<k_2<4$. Observe that this can be done by the non-existence of Forbidden Pattern 2. Now, consider the edge coloring $c:E(K_{I_0}^*) \rightarrow \{0,1,2\}$ of the directed complete graph $K_{I_0}^*$ given by
\[c\big((i',i)\big)=\begin{cases}
    j & \text{ if there exists an edge from } V_{i'}^4 \text{ to } V_{i}^{k_j} \text{ in } \GG_{\cc}\\
    0 & \text{ otherwise }
\end{cases}
\]
The map $c$ is well-defined because of the non-existence of Forbidden Pattern 3. Since $|I_0| = A_{3}(R(K_n^*,L_3))$, it follows from Proposition \ref{avoidancenumberprop} that there exists $I_1 \subseteq I_0$ of size $R(K_n^*,L_3)$ such that
\begin{center}
    for every $i \in I_1$, in the edge colored directed graph $K_{I_0}^*$, the edges from vertices in $I_1$ to the vertex $i$ avoids at least one color, say, $c_i \in \{0,1,2\}$.
\end{center}
Reinterpreting this in the original directed graph $\GG_{\cc}$, we have that, for every distinct $i,i' \in I_1$, there exists no edge from $V_{i'}^4$ to $V_i^{k_{c_i}}$. From now on, we shall refer to the vertex $V_i^{k_{c_i}}$ as the $i^{\text{th}}$ \textit{bottom} vertex.

So far, we have extracted a set of top and bottom vertices
\[\left\{V_i^4,V_{i}^{k_{c_i}}\right\}_{i \in I_1}\]
such that there is no edge from a top vertex $V_i^4$ to any other vertex. At this point, we would like to note the following important observation: For any $\alpha \in V_{i'}^4$ and $\beta \in V_{i}^{k_{c_i}}$ with $i \neq i' \in I_1$, we necessarily have $\cc(\{\alpha,\beta\})=0$ because the non-existence of an edge from $V_{i'}^4$ to $V_{i}^{k_{c_i}}$ implies the non-existence of an edge from $V_{i}^{k_{c_i}}$ to $V_{i'}^4$ as $\cc$ is canonical and the only large subset of $V_{i'}^4$ is the singleton $\{\alpha\}$. It is also clear that we have $\cc(\{\alpha,\beta\})=0$ whenever $\alpha \in V_{i}^4$ and $\beta \in V_{i}^{k_{c_i}}$ because there is no edge from $V_i^4$ to $V_{i}^{k_{c_i}}$ by construction.

In the third and final step, we shall extract a copy of $\omega \cdot n +1$ using the corresponding bottom vertices. For each $i \in I_1$, choose an increasing sequence $(\alpha^i_{n})_{n \in \mathbb{N}}$ of ordinals contained in $V_i^{k_{c_i}}$ which is cofinal in $V_i^{k_{c_i}}$ so that their supremum is the unique ordinal in $V_i^4$. Clearly the set
\[S=\bigcup_{i \in I_1}\{\alpha^i_{n}\}_{n \in \mathbb{N}}\]
is order-isomorphic to $\omega \cdot R(K_n^*,L_3)$. In the case of classical ordinal Ramsey numbers, it is known that $R(\omega \cdot n,3)=\omega \cdot R(K_n^*,L_3)$, see \cite[Theorem 6.4]{CaicedoHilton17}. It now follows from the non-existence of a blue triangle respect to $\cc$ that there is a red homogeneous subset $\widehat{S} \subseteq S$ of order type $\omega \cdot n=\{\omega \cdot \ell+k: 0 \leq \ell <n,\ k \in \omega\}$.

By discarding finitely many elements of $\widehat{S}$ if necessary, we may assume without loss of generality that, for each $0 \leq \ell <n$, the $\ell^{\text{th}}$ copy $\{\omega \cdot \ell+k: \ k \in \omega\}$ of $\omega$ in the set $\widehat{S}$ is completely contained in a bottom vertex, say, $V_{i_{\ell}}^{j_{\ell}}$. Moreover, since the initial sequences were chosen to be cofinal in the corresponding vertices, the bottom vertices $V_{i_{\ell}}^{j_{\ell}}$ and $V_{i_{\ell'}}^{j_{\ell'}}$ are distinct for every $0 \leq \ell \neq \ell' <n$. It is readily verified that the set
\[\bigcup_{\ell=0}^{n-1} V_{i_{\ell}}^{j_{\ell}} \cup V_{i_{\ell}}^4\]
is a red homogeneous copy of $\omega \cdot n +1$, which is a contradiction.
\end{proof}

\section{The lower bound}\label{sectionlowerbound}
In this section, by explicitly constructing colorings that witness various negative closed partition relations, we shall prove the lower bounds in Theorem \ref{maintheorem}.

\subsection{From graphs to canonical colorings}\label{graphstocoloringsection} Recall from Section \ref{directedgraphsection} that each canonical coloring induces a directed graph on a partition of the underlying ordinal. It is only natural to try to reverse this process in order to obtain a canonical coloring from a directed graph on the appropriate partition of the underlying ordinal. It turns out that this can indeed be done, however, the natural context seems to be that of undirected graphs.

Let $\gamma$ be an ordinal and $\GG=(V,E)$ be an undirected graph on a partition $V$ of the ordinal $\gamma$. The graph $\GG$ induces a pair coloring $\cc_{\GG}:[\gamma]^2 \rightarrow \{0,1\}$ given by
\begin{center}
    $\cc_{\GG}(\{\alpha,\beta\})=1$ if and only if the vertices containing $\alpha$ and $\beta$ are adjacent in $\GG$.
\end{center}
It is not difficult to see that the coloring $\cc_{\GG}$ is canonical whenever the partition $V$ is chosen as in Section \ref{directedgraphsection}. From now on, such partitions will be called \textit{canonical partitions}. More specifically, given an ordinal
\[\gamma=\omega^{\alpha_1}+\dots+\omega^{\alpha_n}\]
where $\alpha_1 \geq \dots \geq \alpha_n$ are ordinals, the \textit{canonical partition of $\gamma$} is the partition
\[V=\left\{V_i^{\beta}\right\}_{\substack{1 \leq i < n\\ 0 \leq \beta \leq \alpha_i}} \cup \left\{V_n^{\beta}\right\}_{0 \leq \beta<\alpha_n}  \text{ where } V_i^{\beta}=\{\alpha \in \gamma: CNF_{\gamma}(\alpha)=i \text{ and } CB(\alpha)=\beta\}\]
It turns out that, by constructing graphs $\GG$ on the canonical partition of an ordinal with special properties, one may avoid various prescribed monochromatic homogeneous sets with respect to the coloring $\cc_{\GG}$. Indeed, this has been exactly the approach that is taken in \cite{KayaSaglam21} and \cite[Lemma 5.3]{CaicedoHilton17} to prove negative closed partition relations. Another example result that uses this approach is Theorem \ref{toppartmaintheorem}, which will be proven in Section \ref{sectiontoppart}.

Having illustrated how graphs on canonical partitions of ordinals can be used to find colorings witnessing various negative partition relations, we would like to remark that this approach has limited power because, when we do not put an edge between two vertices, we are coloring \textit{all} pairs of ordinals contained in these vertices to the color red for no reason at all, whereas, coloring \textit{some} of these pairs to blue without creating blue triangles could potentially prevent some red homogeneous copies of ordinals from existing. For example, in the coloring that witnesses negative partition relations in \cite{Mermelstein19}, there are such blue pairs.

In order to remedy this situation, we shall first partition the elements of a canonical partition into further pieces, some of which will be partitioned into further pieces, some of which will be further partitioned and so on. Then we shall induce a coloring using graph-like connections between these pieces. Due to the recursive nature of this process resembling Russian stacking dolls, we shall call such coloring matryoshka colorings.

\subsection{Matryoshka colorings} In this subsection, we shall define the notion of a matryoshka coloring of an ordinal $\gamma<\omega^\omega$. In order to define a matryoshka coloring of $\gamma$, we shall first create what we call matryoshka boxes of $\omega^{\omega}$.

Let $0 \leq j < k$. A \textit{matryoshka box of degree $k$ and rank $j$} is a set $M \subseteq \omega^{\omega}$ such that
\[M=\left\{\alpha \in \omega^{\omega}:\ CB(\alpha)=j \text{ and } \alpha\leq^*\delta\right\}\]
for some $\delta<\omega^{\omega}$ with $CB(\delta)=k-1$. Consider the set
\[\mathcal{M}_k=\{M \subseteq \omega^{\omega}: M \text{ is a matryoshka box of degree }k\}\]
Observe that two distinct matryoshka boxes of the same degree are disjoint. It is straightforward to verify that the relation $\prec$ on the set $\mathcal{M}_k$ given by
\begin{center}
    $M \prec N$ if and only if\\ $\text{rank}(M)<\text{rank}(N)$, or, $\text{rank}(M)=\text{rank}(N)$ and $\sup(M) < \inf(N)$
\end{center}
is a strict well-order relation and that the order type of $(\mathcal{M}_k,\prec)$ is $\omega^{\omega} \cdot k$. From now on, we shall write $M_k(\theta,j)$ to denote the box $\varphi_k(\theta,j)$ where $\varphi_k: \omega^{\omega} \times k \rightarrow (\mathcal{M}_k,\prec)$ is the unique order-isomorphism.

Before we proceed, we would like to note the basic but important observation that every matryoshka box of degree $k+1$ consists of precisely $\omega$-many matryoshka boxes of degree $k$ stacked back to back. We refer the reader to Figure \ref{matryoshkaboxes} for an illustration of matryoshka boxes of degree $2$ up to $5$ inside the ordinal $\omega^4+1$.

We are now ready to define a matryoshka coloring. Let $\cc: [\gamma]^2 \rightarrow \{0,1\}$ be a coloring. Given integers $0 \leq j_2 <j_1<k \leq n$, the coloring $\cc$ is said to have
\begin{itemize}
    \item \textit{forward (respectively, backward) connections of degree $k$ from rank $j_1$ to $j_2$} in the case
\begin{center}
$\cc(\{\alpha,\beta\})=1$ for all $\alpha \in M_k(\theta_1,j_1)$ and $\beta \in M_k(\theta_2,j_2)$\\
whenever $\theta_1<\theta_2$ (respectively, $\theta_1>\theta_2$) and $M_k(\theta_1,j_2)$ and $M_k(\theta_2,j_2)$ are matryoshka boxes of degree $k$ and rank $j_2$ that are contained in the same matryoshka box of degree $k+1$.\\
\end{center}
\item \textit{downward connections of degree $k$ from rank $j_1$ to $j_2$} in the case
\begin{center}
    $\cc(\{\alpha,\beta\})=1$ for all $\alpha \in M_k(\theta,j_1)$ and $\beta \in M_k(\theta,j_2)$\\
    whenever $M_k(\theta,j_1)$ and $M_k(\theta,j_2)$ are matryoshka boxes of degree $k$.\\
\end{center}
\end{itemize}
We shall denote the coloring $\cc$ having forward, backward and downward connections of degree $k$ from rank $j_1$ to $j_2$ by
\[ \fconnect{j_1}{k}{j_2},\ \ \bconnect{j_1}{k}{j_2}\ \ \text{and}\ \ \dconnect{j_1}{k}{j_2}\] respectively. The coloring $\cc$ is said to be \textit{matryoshka} if $\cc$ can be expressed using finitely many forward, backward or downward connections. An illustration of a matryoshka coloring of $\omega^4+1$ expressed by the connections
\[\dconnect{4}{5}{3}\ \ \ \dconnect{4}{4}{3}\ \ \ \dconnect{3}{4}{2}\ \ \ \bconnect{1}{3}{0}\ \ \ \fconnect{1}{2}{0}\]
is given in Figure \ref{matryoshkaboxes} so that the reader may build a visual understanding of matryoshka boxes as well as forward, backward and downward connections.

\begin{figure}[h!]
    \centering    \includegraphics[width=0.95\textwidth]{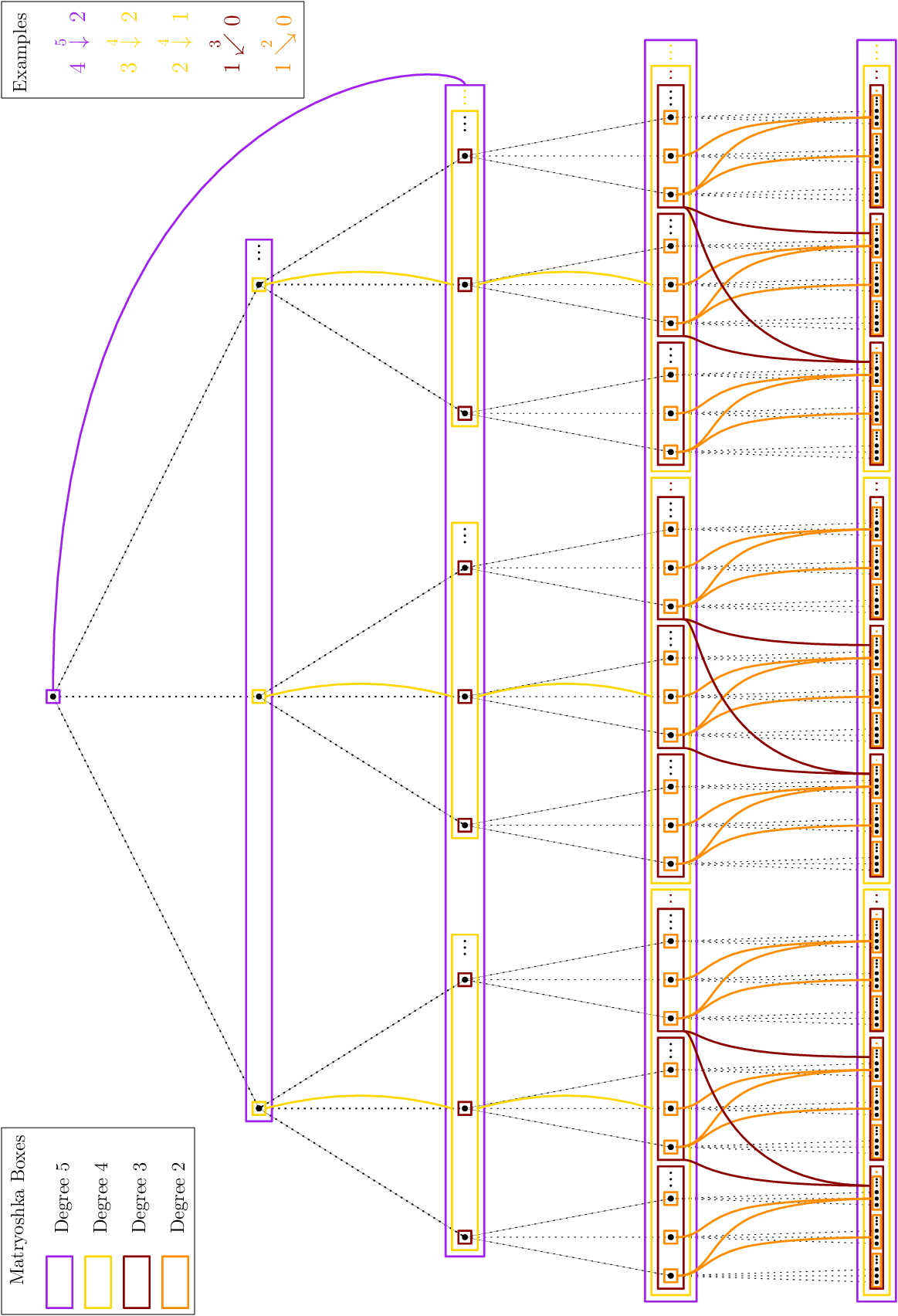}
    \caption{An illustration of matryoshka boxes of degrees $2$ up to $5$ contained in the ordinal $\omega^4+1$, together with some example forward, backward and downward connections}
    \label{matryoshkaboxes}
\end{figure}

\begin{figure}[h!]
    \centering    \includegraphics[width=0.95\textwidth]{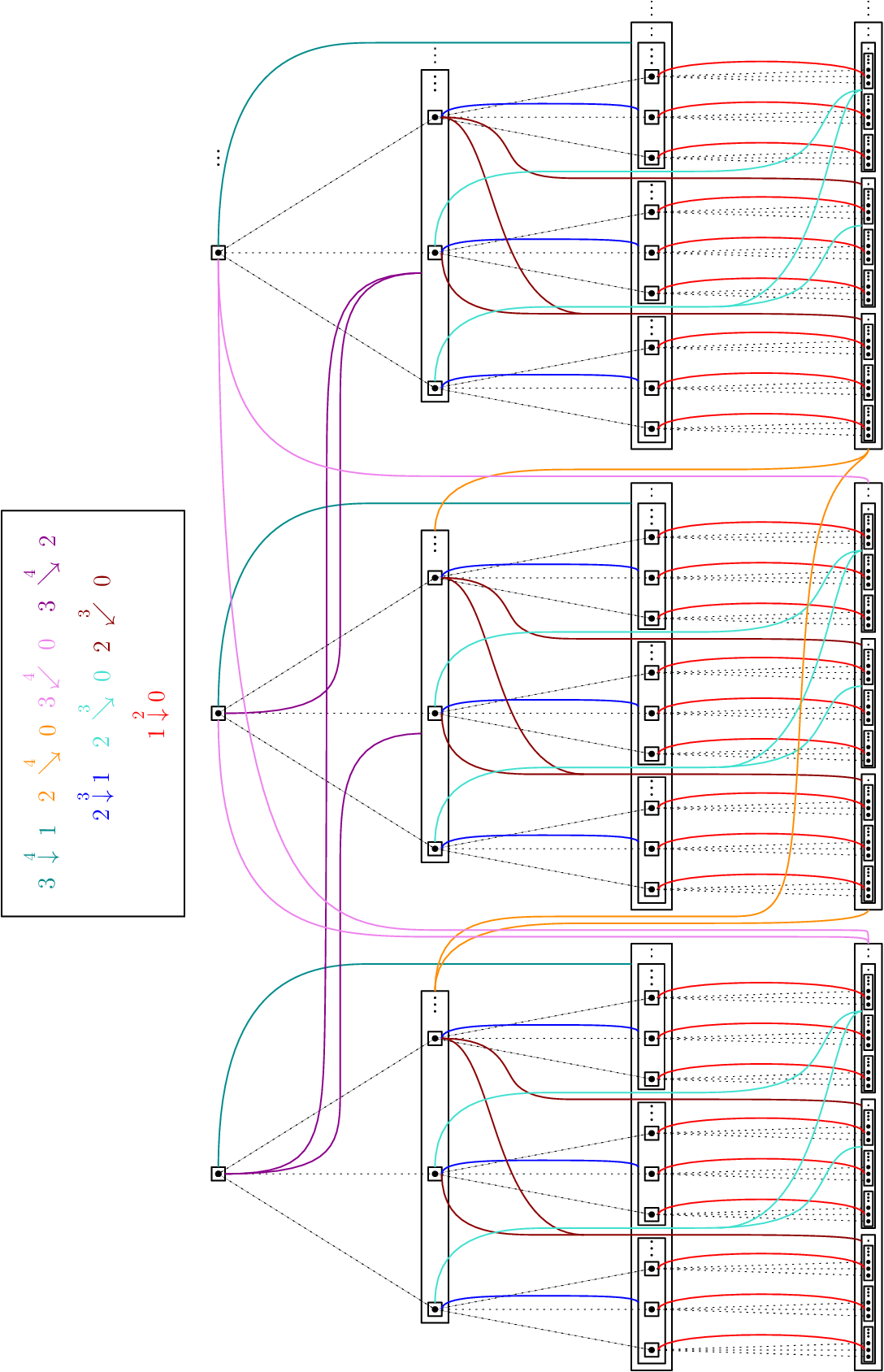}
    \caption{An illustration of the matryoshka coloring $\cc_0$.}
    \label{coloringofw4}
\end{figure}

\subsection{A negative closed partition relation}\label{omegafournegativesection}

In this section, we shall prove the negative closed partition relation $\omega^4 \nrightarrow_{cl} (\omega \cdot 2+1,3)^2$. Consider the matryoshka coloring $\cc_0: \left[\omega^4\right]^2 \rightarrow \{0,1\}$ given by the connections
\[ \dconnect{3}{4}{1}\ \ \ \ \fconnect{2}{4}{0}\ \ \ \  \bconnect{3}{4}{0}\ \ \ \ \fconnect{3}{4}{2}\]
\[\dconnect{2}{3}{1}\ \ \ \ \fconnect{2}{3}{0}\ \ \ \ \bconnect{2}{3}{0}\ \ \ \ \dconnect{1}{2}{0}\]
An illustration of the matryoshka coloring of $\cc_0$ is given in Figure \ref{coloringofw4}. The main proposition of this subsection is the following.
\begin{proposition}\label{omegatofourmatrycoloring} There exist no blue homogeneous copy of $3$ and no red homogeneous closed copy of $\omega \cdot 2+1$ with respect to $\cc_0$.    
\end{proposition}

 We have discovered the coloring $\cc_0$ and verified Proposition \ref{omegatofourmatrycoloring} using a computer search of matryoshka colorings. Since our algorithm and its implementation have not been verified to be bug-free, for the self-containment of the paper, we shall provide a brief proof here explaining the argumentation itself so that an interested reader may manually verify the proof themselves.\footnote{That said, we shall publicly share our code and other matryoshka colorings that we have found, together with an appendix that explains the search and verification algorithms that we used at the project webpage \href{https://users.metu.edu.tr/burakk/124F353}{https://users.metu.edu.tr/burakk/124F353} by October 2026.}

 \begin{proof}[Proof of Proposition \ref{omegatofourmatrycoloring}] Assume towards a contradiction that $\{\alpha,\beta,\delta\} \subseteq \omega^4$ is a blue homogeneous copy of $3$ with $CB(\alpha)=j_1$, $CB(\beta)=j_2$ and $CB(\delta)=j_3$. Since a matryoshka coloring does not color a pair of ordinals of the same Cantor-Bendixson rank to the color blue, we may assume without loss of generality that $0 \leq j_1 < j_2 < j_3 \leq 4$. By the definition of $\cc$, there have to be three connections
 \[j_3 {\overset{\Box}\Box} j_2\ \ j_2 {\overset{\Box}\Box} j_1\ \text{ and }\ j_3 {\overset{\Box}\Box} j_1
\]
that created the blue triangle $\{\alpha,\beta,\delta\}$, where each ${\overset{\Box}\Box}$ is a forward, backward or downward connections of some degree. By construction of $\cc$, the only triples of connections whose ranks are in this form are

\begin{center}
\begin{tabular}{c c c}

$\fconnect{3}{4}{2}$ & $\dconnect{2}{3}{1}$ & $\dconnect{3}{4}{1}$\\
\hline
$\fconnect{3}{4}{2}$ & $\fconnect{2}{4}{0}$ & $\bconnect{3}{4}{0}$\\
\hline
$\fconnect{3}{4}{2}$ & $\fconnect{2}{3}{0}$ & $\bconnect{3}{4}{0}$\\
\hline
$\fconnect{3}{4}{2}$ & $\bconnect{2}{3}{0}$ & $\bconnect{3}{4}{0}$\\
\hline
$\dconnect{3}{4}{1}$ & $\dconnect{1}{2}{0}$ & $\bconnect{3}{4}{0}$\\
\hline
$\dconnect{2}{3}{1}$ & $\dconnect{1}{2}{0}$ & $\fconnect{2}{4}{0}$\\
\hline
$\dconnect{2}{3}{1}$ & $\dconnect{1}{2}{0}$ & $\fconnect{2}{3}{0}$\\
\hline
$\dconnect{2}{3}{1}$ & $\dconnect{1}{2}{0}$ & $\bconnect{2}{3}{0}$\\
\end{tabular}   
\end{center}
Therefore, if we can show that none of these cases are possible, then we will have shown that there exists no blue homogeneous copy of $3$ with respect to $\cc$. We shall eliminate one of the cases, e.g. the third case, as an example. Suppose that the blue triangle $\{\alpha,\beta,\delta\}$ arose as a result of the connections
\[\fconnect{3}{4}{2}\ \ \ \fconnect{2}{3}{0}\ \ \ \bconnect{3}{4}{0}\]
By the first forward connection and the third backward connection, we have that $\alpha \in M_4(\theta_1,3)$, $\beta \in M_4(\theta_2,3)$ and $\delta \in M_4(\theta_3,3)$ for some $\theta_1 < \theta_3 < \theta_2$. Thus the matryoshka boxes of degree $3$ containing $\alpha$ and $\beta$ cannot be contained in the same matryoshka box of degree $4$. But then we cannot have the second connection. Having eliminated this case as an example, we leave it to the reader to eliminate the remaining cases using similar arguments and complete the proof that no blue triangles exist with respect to $\cc_0$.

For the second part of the statement, assume towards a contradiction that there exists a red homogeneous closed copy of $\omega \cdot 2+1$ with respect to $\cc$, say,
\[X=\{\boldord{\alpha}\}_{\alpha \leq \omega \cdot 2} \subseteq \omega^4\]
where $\alpha \mapsto \boldord{\alpha}$ is the unique order-homeomorphism from $\omega \cdot 2+1$ to $X$. By discarding elements if necessary, we may assume without loss of generality that
\[\boldord{k}<^*\boldord{\omega} \text{ and } \boldord{\omega+k+1}<^*\boldord{\omega+\omega} \text{ for all } k<\omega\]
and that
\[CB\left(\boldord{k}\right)=CB\left(\boldord{\ell}\right) \text{ and } CB\left(\boldord{\omega+k+1}\right)=CB\left(\boldord{\omega+\ell+1}\right)\] for all $k,\ell < \omega$, that is, each of the $\omega$-sequences
\[\left\{\boldord{k}:k<\omega\right\} \text{ and }\left\{\boldord{\omega+k}:1\leq k<\omega\right\}\]
is contained in a single level. As in the proof of the non-existence of a blue triangle, the rest of the proof is an exhaustion argument. We go through all possible quintuples of the Cantor-Bendixson ranks of
\begin{itemize}
    \item the first $\omega$-sequence,
    \item the first limit $\boldord{\omega}$,
    \item the second $\omega$-sequence,
    \item the second limit $\boldord{\omega+\omega}$, and
    \item the least common $<^*$-ancestor of the first and second limits.
\end{itemize}
In each case, we run into a contradiction due to the existence of some forward, backward or downward connection. Table \ref{Table1}, \ref{Table2} and \ref{Table3} given below list all these possible cases together with a connection that makes each case not possible. In these tables, rows (respectively, columns) parametrize the Cantor-Bendixson of ranks of the first (respectively, second) $\omega$-sequence and the first (respectively, second) limit.\footnote{The reader who wishes to manually verify each of these cases can use the visualization tool at the project webpage \href{https://users.metu.edu.tr/burakk/124F353}{https://users.metu.edu.tr/burakk/124F353} to build a better understanding of the purpose of each connection.}
\end{proof}

\begin{table}[h!]
    \centering
    \begin{tabular}{|c|c|c|c|c|c|c|c|c|c|c|}
    \hline
    \textbf{CB of $\mathbf{\omega}$-sequence and limit} & $\mathbf{(0,1)} $ & $\mathbf{(0,2)}$&  $\mathbf{(1,2)}$  \\
        \hline
        $\mathbf{(0,1)} $ &$\dconnect{1}{2}{0}$&$\dconnect{1}{2}{0}$ & $\dconnect{1}{2}{0}$\\ \hline
    \end{tabular}
    \vspace{0.2cm}
    \caption{The CB rank of the least common $<^*$-ancestor is $2$}
    \label{Table1}

    \begin{tabular}{|c|c|c|c|c|c|c|c|c|c|c|}
    \hline
        \textbf{CB of $\mathbf{\omega}$-sequence and limit} & $\mathbf{(0,1)} $ & $\mathbf{(0,2)}$ & $\mathbf{(1,2)}$  &$\mathbf{(0,3)}$ & $\mathbf{(1,3)}$&$\mathbf{(2,3)}$  \\
        \hline
        $\mathbf{(0,1)} $ &$\dconnect{1}{2}{0}$&$\dconnect{1}{2}{0}$&$\dconnect{1}{2}{0}$ & $\dconnect{1}{2}{0}$ & $\dconnect{1}{2}{0}$& $\dconnect{1}{2}{0}$\\ \hline
        $\mathbf{(0,2)} $ &$\dconnect{1}{2}{0}$&$\bconnect{2}{3}{0}$&$\dconnect{2}{3}{1}$&$\fconnect{2}{3}{0}$& $\dconnect{3}{4}{1}$& $\bconnect{2}{3}{0}$\\ \hline
        $\mathbf{(1,2)} $ &$\dconnect{1}{2}{0}$&$\dconnect{2}{3}{1}$&$\dconnect{2}{3}{1}$ & $\dconnect{2}{3}{1}$ &$\dconnect{2}{3}{1}$ &$\dconnect{2}{3}{1}$\\ \hline
    \end{tabular}   
    \vspace{0.2cm}
    \caption{The CB rank of the least common $<^*$-ancestor is $3$}
     \label{Table2}
    \begin{tabular}{|c|c|c|c|c|c|c|c|c|c|c|}
    \hline
        \textbf{CB of $\mathbf{\omega}$-sequence and limit} & $\mathbf{(0,1)} $ & $\mathbf{(0,2)}$ & $\mathbf{(1,2)}$ & $\mathbf{(0,3)}$ & $\mathbf{(1,3)}$ & $\mathbf{(2,3)}$  \\
        \hline
        $\mathbf{(0,1)} $ &$\dconnect{1}{2}{0}$&$\dconnect{1}{2}{0}$&$\dconnect{1}{2}{0}$&$\dconnect{1}{2}{0}$&$\dconnect{1}{2}{0}$&$\dconnect{1}{2}{0}$\\ \hline
        $\mathbf{(0,2)} $ &$\dconnect{1}{2}{0}$&$\fconnect{2}{4}{0}$&$\dconnect{2}{3}{1}$&$\fconnect{2}{4}{0}$&$\bconnect{3}{4}{0}$&$\bconnect{3}{4}{0}$\\ \hline
        $\mathbf{(1,2)} $ &$\dconnect{1}{2}{0}$&$\dconnect{2}{3}{1}$&$\dconnect{2}{3}{1}$&$\dconnect{2}{3}{1}$&$\dconnect{2}{3}{1}$&$\dconnect{2}{3}{1}$\\ \hline
        $\mathbf{(0,3)} $ &$\dconnect{1}{2}{0}$& $\fconnect{3}{4}{2}$&$\dconnect{2}{3}{1}$&$\bconnect{3}{4}{0}$&$\bconnect{3}{4}{0}$&$\bconnect{3}{4}{0}$\\ \hline
        $\mathbf{(1,3)} $ &$\dconnect{1}{2}{0}$&$\dconnect{3}{4}{1}$&$\dconnect{2}{3}{1}$&$\dconnect{3}{4}{1}$&$\dconnect{3}{4}{1}$&$\dconnect{3}{4}{1}$\\ \hline
        $\mathbf{(2,3)} $ &$\dconnect{1}{2}{0}$&$\fconnect{2}{4}{0}$&$\dconnect{2}{3}{1}$&$\fconnect{2}{4}{0}$&$\dconnect{3}{4}{1}$&$\fconnect{3}{4}{2}$\\ \hline
    \end{tabular}
    \vspace{0.2cm}
    \caption{The CB rank of the least common $<^*$-ancestor is $4$}
     \label{Table3}
\end{table}

\subsection{The general lower bound} In this section, we will prove the negative closed partition relation $\omega^4 \cdot (n-2)+1 \nrightarrow_{cl} (\omega \cdot n+1,3)^2$
for all integers $n \geq 3$. We first handle the case $n \geq 4$.

Let $n \geq 4$ be an integer and let $\gamma=\omega^4 \cdot (n-2)+1$. Consider the matryoshka coloring $\cc_1:[\gamma]^2 \rightarrow \{0,1\}$ obtained by adding the connections
\[\fconnect{4}{5}{0}\ \ \ \bconnect{4}{5}{0}\ \ \ \dconnect{4}{5}{3}\]
to the matyroshka coloring $\cc_0$ defined in Section \ref{omegafournegativesection}.

Let us intuitively explain how this coloring relates to $\cc_0$. Recall that the tree $T(\omega^4 \cdot i)$ is order-homeomorphic to $\omega^4+1$ for every $1 \leq i \leq (n-2)$. What the coloring $\cc_1$ does is that it colors the pairs in each $T(\omega^4 \cdot i)$ using $\cc_0$ as if this is a standalone copy of $\omega^4+1$. In addition to these, the pair $\{\alpha,\beta\}$ is colored to blue whenever $\alpha$ is the top point $\omega^4 \cdot i$ of the $i^{\text{th}}$-tree in the forest representation of $\gamma$ and $\beta$ is in the third level of the $i^{\text{th}}$-tree, or, in the bottom levels of other trees.

In contrast to $\cc_0$, there \textit{does} exist a red homogeneous copy of $\omega \cdot 2+1$ with respect to $\cc_1$ in each tree in the forest representation of $\gamma$. Let $X \subseteq \omega^4 \cdot (n-2)+1$ be any red homogeneous copy of $\omega \cdot 2 +1$ that is contained in a tree of the form $T(\omega^4 \cdot i)$. As in the proof of Proposition \ref{omegatofourmatrycoloring}, by discarding elements if necessary, we may assume that each $\omega$-sequence converging to a limit point is contained in a single level. It is not difficult to verify that, in the notation of the proof of Proposition \ref{omegatofourmatrycoloring}, the Cantor-Bendixson rank of
\begin{itemize}
    \item the first $\omega$-sequence is $0$,
    \item the first limit $\boldord{\omega}$ is $2$,
    \item the second $\omega$-sequence is $2$,
    \item the second limit $\boldord{\omega+\omega}$ is 4.
\end{itemize}
This simple but important observation allows us to prove the main theorem of this subsection.

\begin{proposition}\label{mainlowerbound} There exist no blue homogeneous copy of $3$ and no red homogeneous closed copy of $\omega \cdot n+1$ with respect to $\cc_1$.    
\end{proposition}
\begin{proof} Since no blue triangles with respect to $\cc_0$ exist, there exist no blue homogeneous copy of $3$ with respect to $\cc_1$ exist as well. Because, otherwise, a pair of ordinals with the same Cantor-Bendixson rank (namely, $0$ or $3$) would have already been colored to blue with respect to $\cc_0$, which is not the case.

Assume towards a contradiction that a red homogeneous closed copy of $\omega \cdot n+1$ exists. Since there are $(n-2)$ trees in the forest representation of $\gamma$ each of which can contain at most two limit points of this red homogeneous closed copy of $\omega \cdot n +1$, at least two of these trees of must contain at least two limit points. Therefore two red homogeneous closed copies of $\omega \cdot 2 +1$ exist in two different trees. By the observation above, for each of these red homogeneous closed copies of $\omega \cdot 2 +1$, the Cantor-Bendixson rank of the first $\omega$-sequence and the second limit point must be $0$ and $4$ respectively. However, in the coloring $\cc_1$, the top level of each tree is connected to the bottom levels of other trees, which leads to a contradiction.
\end{proof}

Finally, for the case $n=3$, the matryoshka coloring obtained by adding the connection
\[\dconnect{4}{5}{3}\]
to the matyroshka coloring $\cc_0$ defined in Section \ref{omegafournegativesection} witnesses the negative partition relation $\omega^4+1 \nrightarrow_{cl} (\omega \cdot 3+1,3)^2$ because there cannot be three limit points contained in the same tree due to these connections. Combining this result with Theorem \ref{mainuppertheorem}, Proposition \ref{omegatofourmatrycoloring} and Proposition \ref{mainlowerbound}, we have now proven Theorem \ref{maintheorem}.

\section{Topological partition ordinals}\label{sectiontoppart}

In this section, we prove Theorem \ref{toppartmaintheorem} and Corollary \ref{toppartordcorollary}. We obtain these results from a simple but powerful lemma regarding the set of Cantor-Bendixson ranks of elements of a closed copy of an ordinal of the form $\omega^{\alpha}+1$.

\subsection{More on the Cantor-Bendixson rank} Let $\alpha$ be an ordinal. The relationship between $CB(\alpha)$ defined in Section \ref{ordinalpreliminaries} and the well-known Cantor-Bendixson rank of a topological space is given as follows: Let $\gamma>\alpha$ be an ordinal. Then $CB(\alpha)$ is the greatest ordinal $\delta$ such that $\alpha \in \gamma^{(\delta)}$, where $\gamma^{(\delta)}$ denotes the $\delta^{\text{th}}$ Cantor-Bendixson derivative of the topological space $\gamma$ endowed with its order topology. For a proof of this fact, see \cite[Proposition 2.3.3]{Hilton16}.

It follows that if $(\alpha_n)_{n \in \mathbb{N}}$ is an increasing sequence of ordinals with $\beta=\sup_{n \in \mathbb{N}} \alpha_n$, then $CB(\beta) \geq \limsup_{n \rightarrow \infty} CB(\alpha_n)$ because, while iterating the Cantor-Bendixson derivative, $\beta$ cannot disappear before all but finitely many of $\alpha_n$'s disappear, being their limit in the order topology. Moreover, by the same reasoning, if the sequence $(CB(\alpha_n))_{n \in \mathbb{N}}$ is eventually constant, then we must have $CB(\beta) > \lim_{n \rightarrow \infty} CB(\alpha_n)$. Having made these two important observations, we now proceed to prove the main lemma of this section.

Let $X$ be a set of ordinals. We define \textit{the Cantor-Bendixson set} of $X$ to be the set $\{CB(\alpha):\alpha \in X\}$ of the Cantor-Bendixson ranks of elements of $X$. Abusing the notation, we shall also denote this set by $CB(X)$. The main lemma of this section is the following.
\begin{lemma}\label{toppartordinalmainlemma} Let $1 \leq \alpha \leq \theta<\omega_1$ and $X \subseteq \omega^{\theta}+1$ be a closed copy of $\omega^{\alpha}+1$. Then the order type of $CB(X)$ is at least $\alpha+1$.    
\end{lemma}
\begin{proof} We shall prove this by transfinite induction on $\alpha \geq 1$. The claim clearly holds for $\alpha=1$ because, for any closed copy $X$ of $\omega+1$, the Cantor-Bendixson rank of the unique limit point of $X$ has to exceed the Cantor-Bendixson rank of some of its predecessors by the observations above and hence, the order type of $CB(X)$ is at least $2$.

Let $1 \leq \alpha < \theta$ and suppose that the claim holds for $\alpha$. Let $X \subseteq \omega^{\theta}+1$ be a closed copy of $\omega^{\alpha+1}+1$, say, $\varphi: X \rightarrow \omega^{\alpha+1}+1$ is an order-homeomorphism. Consider the sets
\begin{align*}
X_0&=\{x \in X: \varphi(x) \leq \omega^{\alpha}\}\\
X_n&=\{x \in X: \omega^{\alpha}\cdot n <\varphi(x) \leq \omega^{\alpha}\cdot (n+1)\} \text{ for every integer } n \geq 1
\end{align*}
Then each $X_n$ is a closed copy of $\omega^{\alpha}+1$. Set $\eta=\sup(X)=\varphi^{-1}\left(\omega^{\alpha+1}\right)$. We claim that there exists $\widehat{n} \in \mathbb{N}$ such that $CB(\delta)<CB(\eta)$ for every $\delta \in X_{\widehat{n}}$.

Assume towards a contradiction that, for each $n \in \mathbb{N}$, we could choose $\delta_n \in X_n$ with $CB(\delta_n) \geq CB(\eta)$. If $CB(\delta_n) > CB(\eta)$ holds for infinitely many $n \in \mathbb{N}$, then we would have
\[CB(\eta)=CB\left(\sup_{n \in \mathbb{N}} \delta _n \right) \geq \limsup_{n \rightarrow \infty} CB(\delta_n)>CB(\eta)\]
by the first observation above. If $CB(\delta_n) = CB(\eta)$ holds for cofinitely many $n \in \mathbb{N}$, then we would have
\[CB(\eta)=CB\left(\sup_{n \in \mathbb{N}} \delta _n \right) > \lim_{n \rightarrow \infty} CB(\delta_n)=CB(\eta)\]
by the second observation above.

Now choose $\widehat{n} \in \mathbb{N}$ as above. Since $X_{\widehat{n}}$ is a closed copy of $\omega^{\alpha}+1$, by the inductive assumption, the order type of $CB(X_{\widehat{n}})$ is at least $\alpha+1$ and hence, the order type of $CB(X)$ is at least $(\alpha+1)+1$. This completes the successor step of the induction.

Let $1 \leq \gamma \leq \theta$ be a limit and suppose that the claim holds for all $1 \leq \alpha<\gamma$.  Let $X \subseteq \omega^{\theta}+1$ be a closed copy of $\omega^{\gamma}+1$, say, $\phi: X \rightarrow \omega^{\gamma}+1$ is an order-homeomorphism. Choose an increasing sequence $(\alpha_n)_{n \in \mathbb{N}}$ with $\sup_{n \in \mathbb{N}} \alpha_n = \gamma$. Consider the sets
\begin{align*}
X_0&=\{x \in X: \phi(x) \leq \omega^{\alpha_0}\}\\
X_n&=\{x \in X: \omega^{\alpha_{n-1}}<\phi(x) \leq \omega^{\alpha_n}\} \text{ for every integer } n \geq 1
\end{align*}
Then each $X_n$ is a closed copy of $\omega^{\alpha_n}+1$. By inductive assumption, the order type of each $CB(X_n)$ is at least $\alpha_n+1$.

Set $\eta=\sup(X)=\phi^{-1}(\omega^{\gamma})$. Imitating the argument in the successor stage, one can show that, for cofinitely many $n \in \mathbb{N}$, we have $CB(\delta)<CB(\eta)$ for all $\delta \in X_n$.  It follows that the order type of $CB(X)=CB\left(\bigcup_{n \in \mathbb{N}} X_n \cup \{\eta\}\right)$ is at least $\gamma+1$. This completes the induction.\end{proof}

\subsection{A necessary condition for topological partition ordinals} An immediate corollary of Lemma \ref{toppartordinalmainlemma} is that, for any $1 \leq \alpha \leq \theta<\omega_1$ and any closed copy $X$ of $\omega^{\alpha}$ inside $\omega^{\theta}$, the order type of $CB(X)$ is at least $\alpha$ because, otherwise, we would not be able to obtain the order type at least $\alpha+1$ after adding the single limit point $\sup(X)$. We are now ready to prove the remaining main results of the paper.
\begin{proof}[Proof of Theorem \ref{toppartmaintheorem}] Suppose that $\theta<R(\alpha,3)$. Consider the canonical partition
\[V=\left\{V_1^{\delta}\right\}_{0 \leq \delta < \theta} \text{ where } V_1^{\delta}=\{\beta \in \omega^{\theta}: CB(\beta)=\delta\}\]
of the ordinal $\omega^{\theta}$. By the definition of the ordinal Ramsey number $R(\alpha,3)$, there exists a coloring $\mathfrak{r}: [\theta]^2 \rightarrow \{0,1\}$ with respect to which there exist no red homogeneous set of order type $\alpha$ and no blue homogeneous set of size $3$. We define the graph $\GG=(V,E)$ by
\[\left\{V_1^{\delta},V_1^{\delta'}\right\}\in E\ \text{ if and only if }\ \mathfrak{r}(\{\delta,\delta'\})=1\]
Consider the coloring $\cc_{\GG}: \left[\omega^{\theta}\right]^2 \rightarrow \{0,1\}$. Since there exist no blue homogeneous set of size $3$ with respect to $\mathfrak{r}$, there exists no blue homogeneous set of size $3$ with respect to $\cc_{\GG}$. If there were a red homogeneous closed copy $X \subseteq \omega^{\theta}$ of $\omega^{\alpha}$ with respect to $\cc_{\GG}$, then, by (the immediate corollary of) Lemma \ref{toppartordinalmainlemma}, the set
\[CB(X)=\left\{\delta \in \theta:\ V_1^{\delta} \cap X \neq \emptyset\right\}\]
would have to be a red homogeneous set of order type at least $\alpha$ with respect to $\mathfrak{r}$, which is a contradiction. Thus $\omega^{\theta}\nrightarrow_{cl} (\omega^{\alpha},3)^2$.\end{proof}

\begin{proof}[Proof of Corollary \ref{toppartordcorollary}] Suppose that $\omega^{\alpha}$ is a topological partition ordinal, that is, $\omega^{\alpha} \rightarrow_{top} (\omega^{\alpha},3)^2$. By \cite[Corollary 2.17]{Hilton16JSL}, we have $\omega^{\alpha} \rightarrow_{cl} (\omega^{\alpha},3)^2$. It follows from Theorem \ref{toppartmaintheorem} that $\alpha \geq R(\alpha,3)$. Hence $\alpha \rightarrow(\alpha,3)^2$, that is, $\alpha$ is a partition ordinal.\end{proof}
\bibliography{references}{}

@article {ErdosRado53,
    AUTHOR = {Erd\"{o}s, P. and Rado, R.},
     TITLE = {A problem on ordered sets},
   JOURNAL = {J. London Math. Soc.},
  FJOURNAL = {The Journal of the London Mathematical Society},
    VOLUME = {28},
      YEAR = {1953},
     PAGES = {426--438},
      ISSN = {0024-6107},
   MRCLASS = {27.2X},
  MRNUMBER = {58687},
MRREVIEWER = {Djuro Kurepa},
       DOI = {10.1112/jlms/s1-28.4.426},
       URL = {https://doi.org/10.1112/jlms/s1-28.4.426},
}

@article {ErdosRado56,
    AUTHOR = {Erd\"{o}s, P. and Rado, R.},
     TITLE = {A partition calculus in set theory},
   JOURNAL = {Bull. Amer. Math. Soc.},
  FJOURNAL = {Bulletin of the American Mathematical Society},
    VOLUME = {62},
      YEAR = {1956},
     PAGES = {427--489},
      ISSN = {0002-9904},
   MRCLASS = {05.0X},
  MRNUMBER = {81864},
MRREVIEWER = {L. Gillman},
       DOI = {10.1090/S0002-9904-1956-10036-0},
       URL = {https://doi.org/10.1090/S0002-9904-1956-10036-0},
}

@article {Baumgartner86,
    AUTHOR = {Baumgartner, James E.},
     TITLE = {Partition relations for countable topological spaces},
   JOURNAL = {J. Combin. Theory Ser. A},
  FJOURNAL = {Journal of Combinatorial Theory. Series A},
    VOLUME = {43},
      YEAR = {1986},
    NUMBER = {2},
     PAGES = {178--195},
      ISSN = {0097-3165},
   MRCLASS = {05A17 (04A20 54A25)},
  MRNUMBER = {867644},
MRREVIEWER = {N. Hindman},
       DOI = {10.1016/0097-3165(86)90059-2},
       URL = {https://doi.org/10.1016/0097-3165(86)90059-2},
}

@incollection {CaicedoHilton17,
    AUTHOR = {Caicedo, Andr\'{e}s Eduardo and Hilton, Jacob},
     TITLE = {Topological {R}amsey numbers and countable ordinals},
 BOOKTITLE = {Foundations of mathematics},
    SERIES = {Contemp. Math.},
    VOLUME = {690},
     PAGES = {87--120},
 PUBLISHER = {Amer. Math. Soc., Providence, RI},
      YEAR = {2017},
   MRCLASS = {03E02 (03E10 54A25 54H05)},
  MRNUMBER = {3656308},
MRREVIEWER = {P\'{e}ter Komj\'{a}th},
       DOI = {10.1090/conm/690},
       URL = {https://doi.org/10.1090/conm/690},
}

@article {Mermelstein20,
    AUTHOR = {Mermelstein, Omer},
     TITLE = {The closed ordinal {R}amsey number {$R^{cl}(\omega^2,3) =
              \omega^6$}},
   JOURNAL = {Proc. Amer. Math. Soc.},
  FJOURNAL = {Proceedings of the American Mathematical Society},
    VOLUME = {148},
      YEAR = {2020},
    NUMBER = {1},
     PAGES = {413--419},
      ISSN = {0002-9939},
   MRCLASS = {03E02 (03E10)},
  MRNUMBER = {4042862},
       DOI = {10.1090/proc/14697},
       URL = {https://doi.org/10.1090/proc/14697},
}

@article {Mermelstein19,
    AUTHOR = {Mermelstein, Omer},
     TITLE = {Calculating the closed ordinal {R}amsey number
              {$R^{cl}(\omega\cdot 2, 3)$}},
   JOURNAL = {Israel J. Math.},
  FJOURNAL = {Israel Journal of Mathematics},
    VOLUME = {230},
      YEAR = {2019},
    NUMBER = {1},
     PAGES = {387--407},
      ISSN = {0021-2172},
   MRCLASS = {03E02 (03E10 54A25)},
  MRNUMBER = {3941152},
MRREVIEWER = {David Jose Fern\'{a}ndez Bret\'{o}n},
       DOI = {10.1007/s11856-019-1827-0},
       URL = {https://doi.org/10.1007/s11856-019-1827-0},
}

@article {Ramsey29,
    AUTHOR = {Ramsey, F. P.},
     TITLE = {On a {P}roblem of {F}ormal {L}ogic},
   JOURNAL = {Proc. London Math. Soc. (2)},
  FJOURNAL = {Proceedings of the London Mathematical Society. Second Series},
    VOLUME = {30},
      YEAR = {1929},
    NUMBER = {4},
     PAGES = {264--286},
      ISSN = {0024-6115},
   MRCLASS = {99-04},
  MRNUMBER = {1576401},
       DOI = {10.1112/plms/s2-30.1.264},
       URL = {https://doi.org/10.1112/plms/s2-30.1.264},
}

@article {KayaSaglam21,
    AUTHOR = {Kaya, Burak and Sa\u{g}lam, Irmak},
     TITLE = {On the closed {R}amsey numbers {$R^{cl}(\omega + n, 3)$}},
   JOURNAL = {Israel J. Math.},
  FJOURNAL = {Israel Journal of Mathematics},
    VOLUME = {245},
      YEAR = {2021},
    NUMBER = {2},
     PAGES = {963--989},
      ISSN = {0021-2172,1565-8511},
   MRCLASS = {03E02 (54A25)},
  MRNUMBER = {4358268},
MRREVIEWER = {Mohammad\ Golshani},
       DOI = {10.1007/s11856-021-2239-5},
       URL = {https://doi.org/10.1007/s11856-021-2239-5},
}

@article {ErdosRado67,
    AUTHOR = {Erd\H{o}s, P. and Rado, R.},
     TITLE = {Partition relations and transitivity domains of binary
              relations},
   JOURNAL = {J. London Math. Soc.},
  FJOURNAL = {The Journal of the London Mathematical Society},
    VOLUME = {42},
      YEAR = {1967},
     PAGES = {624--633},
      ISSN = {0024-6107,1469-7750},
   MRCLASS = {05.04},
  MRNUMBER = {218248},
MRREVIEWER = {Brian\ Rotman},
       DOI = {10.1112/jlms/s1-42.1.624},
       URL = {https://doi.org/10.1112/jlms/s1-42.1.624},
}

@article {LarsonMitchell97,
    AUTHOR = {Larson, Jean A. and Mitchell, William J.},
     TITLE = {On a problem of {E}rd\"{o}s and {R}ado},
   JOURNAL = {Ann. Comb.},
  FJOURNAL = {Annals of Combinatorics},
    VOLUME = {1},
      YEAR = {1997},
    NUMBER = {3},
     PAGES = {245--252},
      ISSN = {0218-0006,0219-3094},
   MRCLASS = {05C55 (05C20 06A07)},
  MRNUMBER = {1630775},
MRREVIEWER = {Sheshayya\ A.\ Choudum},
       DOI = {10.1007/BF02558478},
       URL = {https://doi.org/10.1007/BF02558478},
}

@phdthesis{Hilton16,
  title={Combinatorics of countable ordinal topologies},
  author={Hilton, Jacob Haim},
  year={2016},
  school={University of Leeds}
}

@article {Hilton16JSL,
    AUTHOR = {Hilton, Jacob},
     TITLE = {The topological pigeonhole principle for ordinals},
   JOURNAL = {J. Symb. Log.},
  FJOURNAL = {The Journal of Symbolic Logic},
    VOLUME = {81},
      YEAR = {2016},
    NUMBER = {2},
     PAGES = {662--686},
      ISSN = {0022-4812,1943-5886},
   MRCLASS = {03E02 (03E15 03E35 03E45 03E55)},
  MRNUMBER = {3519451},
MRREVIEWER = {J\"{o}rg\ D.\ Brendle},
       DOI = {10.1017/jsl.2015.45},
       URL = {https://doi.org/10.1017/jsl.2015.45},
}

@article {GalvinLarson74,
    AUTHOR = {Galvin, Fred and Larson, Jean},
     TITLE = {Pinning countable ordinals},
   JOURNAL = {Fund. Math.},
  FJOURNAL = {Polska Akademia Nauk. Fundamenta Mathematicae},
    VOLUME = {82},
      YEAR = {1974/75},
     PAGES = {357--361},
      ISSN = {0016-2736,1730-6329},
   MRCLASS = {04A10},
  MRNUMBER = {360280},
MRREVIEWER = {E.\ C.\ Milner},
       DOI = {10.4064/fm-82-4-357-361},
       URL = {https://doi.org/10.4064/fm-82-4-357-361},
}

@article {Kim95,
    AUTHOR = {Kim, Jeong Han},
     TITLE = {The {R}amsey number {$R(3,t)$} has order of magnitude
              {$t^2/\log t$}},
   JOURNAL = {Random Structures Algorithms},
  FJOURNAL = {Random Structures \& Algorithms},
    VOLUME = {7},
      YEAR = {1995},
    NUMBER = {3},
     PAGES = {173--207},
      ISSN = {1042-9832,1098-2418},
   MRCLASS = {05C55},
  MRNUMBER = {1369063},
MRREVIEWER = {Pavel\ Valtr},
       DOI = {10.1002/rsa.3240070302},
       URL = {https://doi.org/10.1002/rsa.3240070302},
}

@article {Schipperus10,
    AUTHOR = {Schipperus, Rene},
     TITLE = {Countable partition ordinals},
   JOURNAL = {Ann. Pure Appl. Logic},
  FJOURNAL = {Annals of Pure and Applied Logic},
    VOLUME = {161},
      YEAR = {2010},
    NUMBER = {10},
     PAGES = {1195--1215},
      ISSN = {0168-0072,1873-2461},
   MRCLASS = {03E02 (03E35 05C55)},
  MRNUMBER = {2652192},
MRREVIEWER = {J.\ M.\ Henle},
       DOI = {10.1016/j.apal.2009.12.007},
       URL = {https://doi.org/10.1016/j.apal.2009.12.007},
}
\bibliographystyle{amsalpha}

\end{document}